\documentclass[journal,twoside,web]{ieeecolor}
\usepackage{generic}
\usepackage{cite}
\usepackage{amsmath,amssymb,amsfonts}
\usepackage{algorithmic}
\usepackage{graphicx}
\usepackage{textcomp}
\def\BibTeX{{\rm B\kern-.05em{\sc i\kern-.025em b}\kern-.08em
    T\kern-.1667em\lower.7ex\hbox{E}\kern-.125emX}}
\usepackage{booktabs}

\DeclareSymbolFont{bbold}{U}{bbold}{m}{n}
\DeclareSymbolFontAlphabet{\mathbbold}{bbold}

\newcounter{prop}

\newcounter{rmrk}

\newcounter{thm}
\newtheorem{proposition}[prop]{Proposition}

\newtheorem{remark}[rmrk]{Remark}

\newtheorem{theorem}[thm]{Theorem}

\usepackage{accents}
\newcommand{\ubar}[1]{\underaccent{\bar}{#1}}

\usepackage{hyperref}

\usepackage{cite}

\usepackage{mathtools}

\allowdisplaybreaks

\usepackage{tikz}

\usepackage{subcaption}

\usepackage{bm}
\newcommand{\vect}[1]{\boldsymbol{\mathbf{#1}}}

\usepackage{balance}

\begin{document}
\title{Network-Cognizant Time-Coupled Aggregate Flexibility of Distribution Systems Under Uncertainties}
\author{Bai Cui, \IEEEmembership{Member, IEEE}, Ahmed Zamzam, \IEEEmembership{Member, IEEE}, and Andrey Bernstein, \IEEEmembership{Senior Member, IEEE}
\thanks{This work was authored by the National Renewable Energy Laboratory, operated by Alliance for Sustainable Energy, LLC, for the U.S. Department of Energy (DOE) under Contract No. DE-AC36-08GO28308. Thir work was supported by the Laboratory Directed Research and Development Program at the National Renewable Energy Laboratory. The views expressed in the article do not necessarily represent the views of the DOE or the U.S. Government. The U.S. Government retains and the publisher, by accepting the article for publication, acknowledges that the U.S. Government retains a nonexclusive, paid-up, irrevocable, worldwide license to publish or reproduce the published form of this work, or allow others to do so, for U.S. Government purposes. }
\thanks{The authors are with the National Renewable Energy Laboratory, Golden, CO, USA 80401 (e-mails: \texttt{\{bai.cui, ahmed.zamzam,
andrey.bernstein\}@nrel.gov}).}}

\maketitle
\thispagestyle{empty} 

\begin{abstract}
Increasing integration of distributed energy resources (DERs) within distribution feeders provides unprecedented flexibility at the distribution-transmission interconnection. To exploit this flexibility and to use the capacity potential of aggregate DERs, feasible substation power injection trajectories need to be efficiently characterized. 
This paper provides an ellipsoidal inner approximation of the set of feasible power injection trajectories at the substation such that for any point in the set, there exists a feasible disaggregation strategy of DERs for any load uncertainty realization. 
The problem is formulated as one of finding the robust maximum volume ellipsoid inside the flexibility region under uncertainty. Though the problem is NP-hard even in the deterministic case, this paper derives novel approximations of the resulting adaptive robust optimization problem based on optimal second-stage policies. 
The proposed approach yields less conservative flexibility characterization than existing flexibility region approximation formulations. The efficacy of the proposed method is demonstrated on a realistic distribution feeder.
\end{abstract}


\vspace{-0.1in}

\section{Introduction}
\label{sec:introduction}
\IEEEPARstart{S}{ignificant} changes have been witnessed at the power distribution network level as a result of the increasing integration of distributed energy resources (DERs), such as renewable energy sources, energy storage, and flexible demand response assets. Although the flexibility of a single DER is limited, the rapid proliferation of DERs adds unprecedented \emph{aggregate} flexibility to the operation of distribution networks~\cite{kroposki2018autonomous}. This flexibility refers to the ability of a distribution feeder to reduce, shape, shift, and modulate its aggregated power consumption/injection at the substation level. Clearly, coordinated operation between transmission and distribution (T\&D) networks is required to fully exploit this flexibility. 

Although it is practically impossible to centrally optimize the operation of large-scale DERs in real time because of cybersecurity risks and increasing computational complexity, hierarchical optimization techniques have been proposed to exploit the flexibility of DERs in the coupled networks operations~\cite{guerrero2010hierarchical, zhou2019hierarchical }; however, these methods generally suffer from convergence issues and require continuous exchange of information between T\&D operators. A promising alternative is to characterize an efficient and explicit approximation of the aggregate flexibility at the distribution level. The approximation, which is later used in transmission-level operation as a proxy of a detailed distribution system model, renders the distribution-transmission co-optimization secure and tractable. An example of the framework can be found in \cite{chen2020aggregate}.

A plethora of approaches has been proposed to characterize or approximate the flexibility of DERs~\cite{silva2018estimating, zhao2017geometric, nazir2018inner, xu2020quantification, kundu2019scalable, polymeneas2016aggregate}, yet most do not consider the network operational constraints such as voltage magnitude limits. The assumption that network constraints are nonbinding might be valid when aggregating a small population of DERs; however, these constraints become critical and binding for large-scale DER aggregations. To that end, \cite{chen2020aggregate,chen2020leveraging} recently proposed aggregation methods of DERs that consider the network operational constraints. These approaches aim to obtain the maximum-volume hyperbox that can be inscribed inside the aggregate flexibility region of the network. The disadvantage of hyperbox-based approximations is that they provide only the conservative \emph{time-uncoupled} flexibility region by construction, and therefore, they fail to effectively capture time-coupling constraints that govern the operation of inter-temporal dependent devices, such as energy storage units; heating, ventilating, and air-conditioning (HVAC) systems; and electric vehicle charging stations. From another prospective, the flexibility region of networked DER aggregators was approximated in~\cite{Nazir2019} to guarantee the feasibility of upper layer system constraints, which can be viewed as a dual problem to the problem of aggregate flexibility approximation considered here.

Aggregate flexibility is affected by model uncertainty, which is a salient feature of modern distribution systems wherein load volatility abounds. An approach to approximate network-cognizant aggregate flexibility under load uncertainty was proposed in \cite{li2019robust}. The approach, however, accounts for only the reactive power flexibility region for a single time period. Extension of the method to account for coupled real-reactive power appeared in \cite{tan2020estimating}, but the question of extension to the multiperiod flexibility region remains open.

Despite extensive research, the trio of challenges (network-awareness, uncertainty incorporation, and multiperiod flexibility) has not yet been addressed in the literature, to the best of our knowledge.
In this paper, we propose the first efficient formulations for ellipsoidal inner approximation of the network-cognizant multiperiod aggregate flexibility of networked DERs under uncertainty. We define the aggregate flexibility region as the set of multiperiod real power injections at the substation that can be realized by controllable DERs under any uncertainty realizations. The problem of finding an ellipsoidal inner approximation of the flexibility region can be cast as one of finding the maximum volume inscribe ellipsoid (MVE) of a polytopic projection. Despite the NP-hardness of the posed problem, we propose tractable reformulations and tight approximations based on adaptive robust optimization (ARO) formulation with optimal second-stage policies. The ability of the proposed approximation method to capture a larger volume of the aggregate flexibility region than state-of-the-art hyperbox approximations is demonstrated on a real 126-node distribution feeder. 

\emph{Notation:} Vectors and matrices are represented by upright bold letters, whereas scalars are represented by normal ones. The $n\times n$ identity matrix is denoted by $\mathbf{I}_n$. A zero vector or matrix is denoted by $\mathbbold{0}$. $\| \cdot \|$ denotes the Euclidean norm of a real vector.

\section{System Model} \label{sect:model}

In this paper, we consider a multiphase radial distribution network with nodes collected in the set $\mathcal{N} = \{0, 1, \ldots, N\}$. Let $0$ represent the point of common coupling---i.e., the substation---and define the set $\mathcal{N}^+ := \mathcal{N}\backslash\{0\}$. For simplicity, assume that all the buses $l\in \mathcal{N}$ have three phases: $a$, $b$, and $c$. Define the set ${\Phi}=\{a, b, c\}$ and ${\Phi}_{\Delta} = \{ab, bc, ca\}$. We denote the voltage magnitudes at bus $k$ by $\vect{v}_k \in \mathbb{R}^3$. Concatenating the voltage magnitudes at all buses in $\mathcal{N}^+$, we construct the vector $\vect{v}\in\mathbb{R}^{3N}$.
Let the aggregation time horizon be discretized into $T$ periods, and let the duration of each period be $\tau$.
In this work, we consider three types of DERs: energy storage units, photovoltaic (PV) inverters, and HVAC systems. Note that the proposed aggregation framework can be easily extended to account for other types of DERs.

\vspace{-0.05in}

\subsection{Device Model}

\subsubsection{Energy Storage Units}

Let the set of phases connected to energy storage units be denoted by $\mathcal{B}$. For an energy storage unit installed at phase $\phi$ at bus $k$, its power outputs are constrained by:
\begin{align}\label{eq:bat-rate-const}
    & \ubar{P}_{k, \phi}^{(B)} \leq p_{k, \phi}^{(B)}(t) \leq \bar{P}_{k, \phi}^{(B)}, &&\forall t\\
    & q_{k, \phi}^{(B)}(t) = \psi_k^{(B)} p_{k, \phi}^{(B)}(t), && \forall t
\end{align}
where $\ubar{P}_{k, \phi}^{(B)}$ and $\bar{P}_{k, \phi}^{(B)}$ denote the maximum charging and discharging rates, and $\psi_k^{(B)}$ denotes the fixed power factor of the energy storage unit. In addition, the state of charge of the energy storage unit is assumed to satisfy the following linear model and box constraint:
\begin{align}\label{eq:bat-SoC}
    & e_{k, \phi} (t) = \kappa_k e_{k, \phi}(t-1) - \tau p_{k, \phi}^{(B)}(t), && \forall t \\
    & \ubar{e}_{k, \phi} \leq e_{k, \phi}(t) \leq E_{k, \phi}, && \forall t
\end{align}
where $e_{k, \phi} (t)$ denotes the battery state of charge; $0 < \kappa_k \le 1$ is the storage efficiency factor; and $\ubar{e}_{k, \phi}$ and $E_{k, \phi}$ denote the minimum allowed state of charge and the capacity of the energy storage unit, respectively.

\subsubsection{PV Inverters}

We denote the set of phases with installed PV inverters by $\mathcal{R}$. For a PV unit installed at phase $\phi$ at bus $k$, the active and reactive power injections at time $t$ are constrained as follows:
\begin{align}\label{eq:ren-const}
& 0 \leq p_{k, \phi}^{(R)}(t) \leq \bar{P}_{k, \phi}^{(R)}(t), && \forall t \\
& q_{k, \phi}^{(R)}(t) = \psi_k^{(R)} p_{k, \phi}^{(R)}(t) && \forall t
\end{align}
where $\bar{P}_{k, \phi}^{(R)}(t)$ denotes the available active power at this PV unit at time $t$, and $\psi_k^{(R)}$ denotes the fixed power factor of the PV unit.

\subsubsection{HVAC Systems}

We denote the set of phases with HVAC systems by $\mathcal{H}$. For an HVAC system installed at phase $\phi$ at bus $k$, the active and reactive power injections at time $t$ are constrained as follows:
\begin{align} \label{eq:hvac-const}
& 0 \leq p_{k, \phi}^{(h)}(t) \leq \bar{P}_{k, \phi}^{(h)}(t), \quad q_{k, \phi}^{(h)}(t) = \psi_k^{(h)} p_{k, \phi}^{(h)}(t) && \hspace{-0.2in} \forall t \\
& \ubar{H}_k \le H_k^{\mathrm{in}}(t) \le \bar{H}_k, && \hspace{-0.2in} \forall t \\
& H_k^{\mathrm{in}}(t) = H_k^{\mathrm{in}}(t-1) + \alpha_k (H_k^{\mathrm{out}}(t) - H_k^{\mathrm{in}}(t-1)) \nonumber\\
& \hspace{1.5in} + \tau \beta_k p_k(t)^{(h)}(t), && \hspace{-0.2in} \forall t \label{eq:hvac-const:c}
\end{align}
where $\bar{P}_{k, \phi}^{(h)}(t)$ denotes the active power capacity of the HVAC system at time $t$, $\psi_k^{(h)}$ denotes the fixed power factor, and $H^{\mathrm{in}}(t)$ and $H^{\mathrm{out}}(t)$ are the indoor and outdoor temperatures at time $t$. Equation \eqref{eq:hvac-const:c} depicts the indoor temperature dynamics, where $\alpha_k$ and $\beta_k$ are the parameters specifying the thermal characteristics of the building and the environment. The model details can be found in \cite{li2011optimal}.

\subsection{Uncertain Load Model}

We suppose the model uncertainty comes solely from uncontrollable loads in the network. The uncontrollable loads are classified into three categories based on their daily average real power consumption as: residential (less than $10$ kW), commercial (between $10$ and $100$ kW), and industrial (more than $100$ kW). The load variations within each category are assumed to be similar and share a common uncertainty variable for each time step. To be specific, with given uncertainty level $\delta$, the real power demand $p_{k, \phi}^{(L)}(t)$ of an uncontrollable load at node $k$, phase $\phi$, and time $t$ with nominal power $p_{k,\phi}^{L_0}(t)$ is modeled as:
\begin{align} \label{eq:uncertain}
    p_{k, \phi}^{(L)}(t) = 
    \begin{cases}
        p_{k,\phi}^{L_0}(t) (1 + \delta \zeta^{\mathrm{res}}_t), & \text{if } i \text{ is residential} \\
        p_{k,\phi}^{L_0}(t) (1 + \delta \zeta^{\mathrm{com}}_t), & \text{if } i \text{ is commercial} \\
        p_{k,\phi}^{L_0}(t) (1 + \delta \zeta^{\mathrm{ind}}_t), & \text{if } i \text{ is industrial}
    \end{cases}
\end{align}
where: 
\begin{align} \label{eq:zetat}
    \vect{\zeta}_t = \begin{bmatrix} \zeta^{\mathrm{res}}_t & \zeta^{\mathrm{com}}_t & \zeta^{\mathrm{ind}}_t \end{bmatrix}^\top \text{  and  } \|\vect{\zeta}_t\| \le 1.
\end{align}
For future reference, let $\vect{\zeta} \in \mathbb{R}^{3T}$ be the concatenation of all $\vect{\zeta}_t, t = 1, \ldots, T$. In addition, the uncontrollable loads are assumed to have constant power factor. The proposed flexibility analysis approach is generic enough to handle more general device models and uncertainty sets as long as they fit in the general model into Section \ref{sect:network}, but they are not pursued here because of space limitations.

\subsection{Network Model} \label{sect:network}

The total power injection at time $t$, phase $\phi \in \Phi \cup \Phi_\Delta$, and node $k \in \mathcal{N}^+$ is given by:
\begin{align}\label{eq:nodal-inj}
    p_{k, \phi}(t) &= p_{k, \phi}^{(B)}(t) - p_{k, \phi}^{(R)}(t) - p_{k, \phi}^{(h)}(t) - p_{k, \phi}^{(L)}(t) \\
    q_{k, \phi}(t) &= q_{k, \phi}^{(B)}(t) - q_{k, \phi}^{(R)}(t) - q_{k, \phi}^{(h)}(t) - q_{k, \phi}^{(L)}(t)
\end{align}
The respective device power injection $p_{k,\phi}^{(\cdot)}$, $q_{k,\phi}^{(\cdot)}$ is zero when the specific type of device is not present at the phase. Then, we collect the active (reactive) power injections from the delta and wye phases at all buses $k \in \mathcal{N}^+$ at all time slots $t \in \{1, \ldots, T\}$ in the vectors ${\vect{p}}_{\Delta}$ (${\vect{q}}_{\Delta}$) and ${\vect{p}}_{Y}$ (${\vect{q}}_{Y}$), respectively. 

Then, we can use the linear power flow model developed in~\cite{bernstein2017linear} to approximate the voltage magnitude as follows:
\begin{align}\label{eq:power-flow-voltMag}
    {\bf v} = {\bf M}_{\Delta}^{(p)} {\vect{p}}_{\Delta} + {\bf M}_{\Delta}^{(q)} {\vect{q}}_{\Delta} + {\bf M}_{Y}^{(p)} {\vect{p}}_{Y} + {\bf M}_{Y}^{(q)} {\vect{q}}_{Y} + \tilde{\bf v}
\end{align}
where ${\bf v}\in \mathbb{R}^{3NT}$ collects the voltage magnitudes at all phases $\phi \in \Phi$ at all buses $k \in \mathcal{N}^+$ for all time steps $t \in \{1, \ldots, T\}$. Also, $\tilde{\bf v}$ is the zero-injection voltage of the network. The application of the linear power flow model is justified because of the presence of stringent voltage magnitude constraints in the distribution system \cite{bernstein2017linear}. The voltage magnitudes ${\bf v}$ are constrained to satisfy the lower and upper voltage limits as follows:
\begin{align} \label{eq:voltMag-const}
    \ubar{\bf v} \leq {\bf v} \leq \bar{\bf v}.
\end{align}
In addition, the net power injections at the substation (Node $0$) are given by:
\begin{align}
       {\vect{p}}_0 = {\bf G}_{\Delta}^{(p)} {\vect{p}}_{\Delta} + {\bf G}_{\Delta}^{(q)} {\vect{q}}_{\Delta} + {\bf G}_{Y}^{(p)} {\vect{p}}_{Y} + {\bf G}_{Y}^{(q)} {\vect{q}}_{Y} + {\bf c} 
\end{align}
where the vector ${\vect{p}}_0$ collects the net injections at all phases at the substation for all time instants. 

We define a vector ${\vect{p}} \in \mathbb{R}^{nT}$ that collects the control variables of individual controllable devices at all time steps, which include $p_{k, \phi}^{(B)} (t)$ and $q_{k, \phi}^{(B)} (t)$ for all phases in $\mathcal{B}$, $p_{k, \phi}^{(R)} (t)$ and $q_{k, \phi}^{(R)} (t)$ for all phases in $\mathcal{R}$, and $p_{k, \phi}^{(h)} (t)$ and $q_{k, \phi}^{(h)} (t)$ for all phases in $\mathcal{H}$. By eliminating intermediate variables, all constraints~\eqref{eq:bat-rate-const}--\eqref{eq:voltMag-const} can be written compactly as:
\begin{align} \label{eq:LPfeasibility}
    \vect{W} {\vect{p}} \le \vect{z}(\vect{\zeta}),
\end{align}
where $\vect{W}$ is a constant constraint matrix capturing the device operational constraints and the network voltage magnitude constraints for all time steps. The vector $\vect{z}$ incorporates the uncontrollable loads and is therefore an affine function of the load uncertainty vector $\vect{\zeta}$.

In addition, the substation real power and the outputs of DERs are related by:
\begin{align} \label{eq:aggragation}
    {\vect{p}}_0 = \vect{D}{\vect{p}} + \vect{b}(\vect{\zeta})
\end{align}
where $\vect{D} \in \mathbb{R}^{T \times nT}$ and $\vect{b} \in \mathbb{R}^T$ model the dependency of the substation injections on the devices operations, i.e., the aggregation model. Note that $\vect{D}$ is a constant matrix, whereas $\vect{b}$ depends affinely on the load uncertainty vector $\vect{\zeta}$.

\section{Ellipsoidal Inner Approximation of Aggregate Flexibility} \label{sect:main}
\subsection{Problem Statement and Reformulation}

Given any load uncertainty $\|\vect{\zeta}_t\| \le 1, \forall t$, the set of $\vect{p}_0$ that admits the disaggregation strategy $\vect{p}(\vect{p}_0, \vect{\zeta})$ satisfying \eqref{eq:LPfeasibility}--\eqref{eq:aggragation} is polytopic. The aggregate flexibility region, which is the intersection of such sets under all $\vect{\zeta}$, is also convex. Because explicit characterization of the aggregate flexibility region is hard, this paper aims to identify an MVE $\mathcal{E} = \{ \vect{p}_0 :\; \vect{p}_0 = \vect{E}\vect{\xi} + \vect{e}, \|\vect{\xi}\| \le 1 \} \subset \mathbb{R}^T$ with $\vect{E} \in \mathbb{R}^{T\times T} \succeq 0$ and $\vect{e} \in \mathbb{R}^T$, parametrized by $\vect{\xi} \in \mathbb{R}^T$, that inscribes the aggregate flexibility region. Because the volume of $\mathcal{E}$ is proportional to $\det(\vect{E})$, the problem can be formulated as:
\begin{multline} \label{eq:ARO}
    \max_{\vect{e}, \vect{E} \succeq 0} \Big\{ \log\det \vect{E} :\; \forall \|\vect{\xi}\| \le 1, \forall \|\vect{\zeta}_t\| \le 1, t = 1, \ldots, T \\
    \exists \vect{p} \text{ s.t. }
    \vect{E}\vect{\xi} + \vect{e} = \vect{D}\vect{p} + \vect{b}(\vect{\zeta}), \vect{W}\vect{p} \le \vect{z}(\vect{\zeta}) \Big\},
\end{multline}
where we maximize the volume of the ellipsoid $\mathcal{E}$ such that for any point in $\mathcal{E}$ and any uncertainty realization $\vect{z}(\vect{\zeta}), \vect{b}(\vect{\zeta})$, there is a corresponding disaggregated power $\vect{p}$ satisfying the constraint set $\vect{W}\vect{p} \le \vect{z}(\vect{\zeta})$.

Note that problem \eqref{eq:ARO} is an ARO problem containing both equality and inequality linear constraints. In general, equality constraints in robust optimization problems are hard to deal with and should be eliminated whenever possible \cite{gorissen2015practical}. So, we first discuss how to eliminate the equality constraints $\vect{E}\vect{\xi} + \vect{e} = \vect{D}\vect{p} + \vect{b}(\vect{\zeta})$ relating to the ellipsoidal characterization of substation power and individual load power. 

Because the aggregated power $p_{0,t}$ is only a function of individual load powers at time $t$, the $t$-th row of matrix $\vect{D}$ has only nonzero entries for columns corresponding to time $t$, so the rows of $\vect{D}$ are linearly independent. It follows that $\vect{D} \in \mathbb{R}^{T\times nT}$ has rank $T$. Let $\vect{b}^1, \ldots, \vect{b}^{nT}$ be orthogonal bases of $\mathbb{R}^{nT}$, such that $\vect{b}^{T+1}, \ldots, \vect{b}^{nT}$ span the null-space of $\vect{D}$; therefore, every $\vect{p} \in \mathbb{R}^{nT}$ can be written uniquely as $\vect{p} = \sum_{i=1}^{T} x_i \vect{b}^i + \sum_{i=1}^{(n-1)T} y_i\vect{b}^{i+T} = \vect{B}_1\vect{x} + \vect{B}_2\vect{y}$ for real numbers $x_1, \ldots, x_{T}, y_1, \ldots, y_{(n-1)T}$ where:
\begin{align}
    \vect{B}_1 &= \begin{bmatrix} \vect{b}^1 & \ldots & \vect{b}^T \end{bmatrix}
    , \;\; \vect{B}_2 = \begin{bmatrix} \vect{b}^{T+1} & \ldots & \vect{b}^{nT}  \end{bmatrix}.
\end{align}
Constraint \eqref{eq:aggragation} can be rewritten in the new coordinates $\vect{b}^1, \ldots, \vect{b}^{nT}$ as:
\begin{align} \label{eq:LPfeasibility:reform}
    \vect{p}_0 = \vect{D}( \vect{B}_1 \vect{x} + \vect{B}_2 \vect{y}) + \vect{b}(\vect{\zeta}) = \tilde{\vect{D}} \vect{x} + \vect{b}(\vect{\zeta}),
\end{align}
where $\tilde{\vect{D}} := \vect{D}\vect{B}_1 \in \mathbb{R}^{T \times T}$ is nonsingular, and $\vect{D}\vect{B}_2\vect{y}$ vanishes because $\vect{B}_2$ spans the null-space of $\vect{D}$. Similarly, we can replace $\vect{p}$ by $\vect{B}_1\vect{x} + \vect{B}_2\vect{y}$ in \eqref{eq:LPfeasibility}, which becomes:
\begin{align} \label{eq:aggragation:reform:1}
    \vect{W}_1\vect{x} + \vect{W}_2\vect{y} \le \vect{z}(\vect{\zeta}). 
\end{align}
where $\vect{W}_1 := \vect{W}\vect{B}_1 \in \mathbb{R}^{m \times T}$, $\vect{W}_2 := \vect{W}\vect{B}_2 \in \mathbb{R}^{m \times (n-1)T}$, and $m$ denote the number of inequality constraints in \eqref{eq:LPfeasibility}. The equality constraint \eqref{eq:LPfeasibility:reform} can be subsequently eliminated by substituting $\vect{x} = \tilde{\vect{D}}^{-1} (\vect{p}_0 - \vect{b}(\vect{\zeta}))$ in \eqref{eq:aggragation:reform:1}, which yields:
\begin{align} \label{eq:LPfeasibility:reform:2}
    \vect{W}_1 \tilde{\vect{D}}^{-1} (\vect{p}_0 - \vect{b}(\vect{\zeta})) + \vect{W}_2\vect{y} \le \vect{z}(\vect{\zeta}). 
\end{align}
Constraint \eqref{eq:LPfeasibility:reform:2} is equivalent to \eqref{eq:LPfeasibility}--\eqref{eq:aggragation} in the sense that, for given $\|\vect{\zeta}_t \| \le 1$, $t = 1, \ldots, T$, $(\vect{p}_0, \vect{p})$ satisfies constraints \eqref{eq:LPfeasibility}--\eqref{eq:aggragation} if and only there exists $(\vect{p}_0, \vect{y})$ satisfying \eqref{eq:LPfeasibility:reform:2} for some $\vect{y}$. For brevity, we rewrite \eqref{eq:LPfeasibility:reform:2} as:
\begin{align}
    \vect{W}_1 \tilde{\vect{D}}^{-1} \vect{p}_0 + \vect{W}_2\vect{y} \le \vect{u}(\vect{\zeta}),
\end{align}
where $\vect{u}(\vect{\zeta}) = \vect{z}(\vect{\zeta}) + \vect{W}_1 \tilde{\vect{D}}^{-1} \vect{b}(\vect{\zeta})$. For future reference, because $\vect{z}$ and $\vect{b}$ are affinely dependent on $\vect{\zeta}$, we denote the explicit affine relationship between $\vect{u}$ and $\vect{\zeta}$ by $\vect{u} := \vect{\Theta}\vect{\zeta} + \vect{\nu} = \sum_{t=1}^T\vect{\Theta}_t \vect{\zeta}_t + \vect{\nu}$, and we denote the corresponding uncertainty set as $\mathcal{U} = \{ \vect{u} :\;  \vect{u} = \sum_{t=1}^T\vect{\Theta}_t \vect{\zeta}t + \vect{\nu}, \; \| \vect{\zeta}_t \| \le 1, t = 1, \ldots, T \}$. 

This discussion leads to the following proposition on the reformulation of \eqref{eq:ARO}:
\begin{proposition} 
    Problem \eqref{eq:ARO} can be reformulated as:
    \begin{multline} \label{eq:6}
        \max_{\vect{e}, \vect{E} \succeq 0} \Big\{ \log\det \vect{E} :\; \forall \|\vect{\xi}\| \le 1, \forall \vect{u} \in \mathcal{U},
        \exists \vect{y} \text{ s.t. } \\
        \vect{W}_1 \tilde{\vect{D}}^{-1} (\vect{E} \vect{\xi} + \vect{e}) + \vect{W}_2\vect{y} \le \vect{u} \Big\},
    \end{multline}
    where the individual DER control vector $\vect{p}$ can be recovered as $\vect{p} = \vect{B}_1 \tilde{\vect{D}}^{-1} (\vect{p}_0 - \vect{b}(\vect{\zeta})) + \vect{B}_2 \vect{y}$ given $\vect{p}_0$, $\vect{y}$, and $\vect{\zeta}$.
\end{proposition}

The problem \eqref{eq:6} is one of finding MVE in a polytopic projection under uncertainty. Although finding the MVE inscribed in a polytope is relatively easy, the same does not hold for \eqref{eq:6}. As noted in \cite{zhen2018computing}, this problem is generally intractable even when no uncertainties are present---i.e., when $\mathcal{U}$ is a singleton---because deriving an explicit description of a projected polytope is NP-hard. We extend the tractable reformulations therein to the case with ellipsoidal uncertainty set $\mathcal{U}$, and we present novel reformulations of \eqref{eq:6} based on quadratic and affine policies in the following section.

\subsection{Policy-Based Convex Approximations}

We first discuss a convex restriction of \eqref{eq:6} based on quadratic policy where we restrict the second-stage variable $\vect{y}$ to be a quadratic function of the uncertainty $\vect{\eta} := \begin{bsmallmatrix} \vect{\xi} \\ \vect{\zeta} \end{bsmallmatrix}$ as:
\begin{align} \label{eq:quadratic}
    \vect{y} = \begin{Bmatrix} \vdots \\ \vect{\eta}^\top \vect{Q}_i \vect{\eta} \\ \vdots \end{Bmatrix} + \vect{L} \vect{\eta} + \vect{c},
\end{align}
where $\vect{Q}_i, i = 1, \ldots, (n-1)T $, $\vect{L}$, and $\vect{c}$ are the policy parameters. We use $\vect{w}_{1,i}$, $\vect{w}_{2,i}$, and $\vect{\theta}_{i}$ to denote the $i$-th row of $\vect{W}_1$, $\vect{W}_2$, and $\vect{\Theta}$, respectively. Using the approximate S-Lemma \cite[Thm. B.3.1]{ben2009robust}, the following theorem provides a tractable approximation of \eqref{eq:6} under quadratic policy \eqref{eq:quadratic}. 

\begin{theorem} \label{thm:quadratic}
    Let $ \vect{\rho}^i := ([ \vect{w}_{1,i}\tilde{\vect{D}}^{-1} \vect{E}, \; -\boldsymbol{\theta}_{i} ] + \vect{w}_{2,i}\vect{L}) / 2$ and: 
    \begin{align}
        \vect{J}_i := 
        \begin{bmatrix} 
            \lambda_{1,i}\vect{I}_T & \mathbbold{0} & \cdots & \mathbbold{0} \\
            \mathbbold{0} & \lambda_{2,i} \vect{I}_{N_u} & \ddots & \vdots \\
            \vdots & & \ddots & \mathbbold{0} \\
            \mathbbold{0} & \cdots & & \lambda_{T+1,i} \vect{I}_{N_u}
        \end{bmatrix}
    \end{align}
    for $i = 1, \ldots, m$, then the problem:
    \begin{subequations} \label{eq:quadratic:reform}
    \begin{align}
        &\max_{\substack{\vect{E} \succeq 0, \vect{e}, \vect{c},\vect{L}, \\ \{\vect{Q}_i\}_{i=1}^{m}, \{\lambda_{k,i}\}_{k=0, i=1}^{T+1, m} }} \quad  \log \det \vect{E} \\
        &\quad \text{s.t.} \quad  (\forall i = 1, \ldots, m) \nonumber\\
        &  \sum_{k=0}^{T+1} \lambda_{k,i} + \vect{w}_{1,i} \tilde{\vect{D}}^{-1} \vect{e} + \vect{w}_{2,i} \vect{c} - \nu_i \le 0, \label{eq:quadratic:reform:b} \\
        &  \begin{bmatrix} 
        \lambda_{0,i} & \mathbbold{0} \\ 
        \mathbbold{0} & \vect{J}_i \end{bmatrix} \succeq \begin{bmatrix}
        0 & \vect{\rho}^i \\
        (\vect{\rho}^i)^\top & \sum\limits_{j=1}^{(n-1)T} {W_{2,ij}\vect{Q}_j} 
        \end{bmatrix}, \label{eq:quadratic:reform:c} \\
        & \lambda_{k,i} \ge 0, \; \qquad k = 0, \ldots, T+1, \label{eq:quadratic:reform:d}
    \end{align}
    \end{subequations}
    is an approximation of \eqref{eq:6} with quadratic policy \eqref{eq:quadratic} with tightness factor at most $9.19 \sqrt{\ln{(T+1)}}$. That is, the feasible set of \eqref{eq:quadratic:reform} projected on $(\vect{E}, \vect{e})$-space is a superset of that of \eqref{eq:6} when the bounds on the Eucliean norm of $\vect{\xi}$ and $\vect{\zeta}_t, \forall t$ are relaxed from $1$ to $9.19 \sqrt{\ln{(T+1)}}$ in \eqref{eq:6}.
\end{theorem}
\begin{proof}
    Denote the RHS of \eqref{eq:quadratic:reform:c} by $\tilde{\vect{Q}}_i$, then the ARO \eqref{eq:6} with quadratic policy \eqref{eq:quadratic} can be reformulated as the following robust optimization problem by plugging \eqref{eq:quadratic} into \eqref{eq:6}:
    \begin{subequations}
    \begin{align}
        &\max_{\vect{E} \succeq 0, \vect{e}, \vect{c}, \vect{L}, \{\vect{Q}_i\}_{i=1}^{m}} \quad  \log \det \vect{E} \\
        &\quad \text{s.t.} \qquad 
         (\forall i =1, \ldots, m) \nonumber\\
        &\vect{w}_{1,i} \tilde{\vect{D}}^{-1} \vect{e}\! +\! \vect{w}_{2,i} \vect{c} - \vect{\nu}_i + \max_{\substack{y = 1, \|\vect{\xi}\|_2 \le 1 \\ \| \vect{\zeta}_k \|_2 \le 1}} \left\{ \begin{bmatrix} y \\ \vect{\eta} \end{bmatrix}^\top\!\! \tilde{\vect{Q}}_i \begin{bmatrix} y \\ \vect{\eta} \end{bmatrix} \right\}\! \le\! 0, 
        \label{eq:ref1:b}
    \end{align}
    \end{subequations}
    The key observation is that the maximization problem in \eqref{eq:ref1:b} (which we denote by (P)) admits an exact reformulation by relaxing $y=1$ to $y^2 \le 1$ (we refer to the relaxed problem by (R)). To see this, suppose $(y^*, \vect{\eta}^*)$ is an optimal solution of (P) such that $|y^*| < 1$. The optimal value of (P) is given by $2 (\vect{\rho}^i)^\top \vect{\eta}^* y^* + {\vect{\eta}^*}^\top (\sum_{j=1}^{(n-1)T} {W_{2,ij} \vect{Q}_j}) \vect{\eta}^*$. Without loss of generality, we may assume $2 (\vect{\rho}^i)^\top \vect{\eta}^*$ is nonnegative because $(-y^*, -\vect{\eta}^*)$ is also an optimal solution. It then follows that when $2 (\vect{\rho}^i)^\top \vect{\eta}^* > 0$, $2 (\vect{\rho}^i)^\top \vect{\eta}^* y^* < 2 (\vect{\rho}^i)^\top \vect{\eta}^*$, which contradicts the fact that $(y^*, \vect{\eta}^*)$ is optimal; and when $2 (\vect{\rho}^i)^\top \vect{\eta}^* = 0$, $(1, \vect{\eta}^*)$ is also optimal. 
    
    We can therefore replace (P) by (R), and \eqref{eq:quadratic:reform} follows by applying the approximate $S$-Lemma on (R). The tightness factor comes from the matrix cube theorem \cite{ben2003extended} (cf. (B-3) of \cite[Thm. B.3.1.]{ben2009robust}).
\end{proof}

Alternatively, one may also impose simpler policies, such as an affine policy, which takes the following form:
\begin{align} \label{eq:affine}
    \vect{y} = \vect{K} \vect{\xi} + \sum_{t=1}^T \vect{L}_t \vect{\zeta}_t + \vect{\gamma},
\end{align}
where the decision variables are $\vect{K} \in \mathbb{R}^{m \times T}$, $\vect{L}_i \in \mathbb{R}^{m \times N_u}$ for $i = 1, \ldots, T$, and $\vect{\gamma} \in \mathbb{R}^m$. Though it is more restrictive than the quadratic policy, problem \eqref{eq:6} with affine policy \eqref{eq:affine} admits an exact tractable reformulation, which is computationally much more efficient than its quadratic policy counterpart. The tractable reformulation of \eqref{eq:6} with affine policy is given in the following theorem:
\begin{theorem} \label{thm:affine}
    The problem: 
    \begin{subequations} \label{eq:affine:reform}
    \begin{align}
        & \max_{ \vect{E} \succeq 0, \vect{e}, \vect{K}, \{ \vect{L}_t\}_{t=1}^{T}, \vect{\gamma}, \{\alpha_i\}_{i=1}^{m}} \quad \log \det \vect{E} \\
        & \quad \text{s.t.} \qquad \quad (\forall i = 1, \ldots, m)\nonumber\\ &\| \vect{w}_{1,i}\tilde{\vect{D}}^{-1}\vect{E} + \vect{w}_{2,i}\vect{K} \| + \sum_{k=1}^T \| \vect{w}_{2,i}\vect{L}_k - \vect{\theta}_{k,i} \| \leq \alpha_i,\\
        & \alpha_i+ \vect{w}_{1,i}\tilde{\vect{D}}^{-1}\vect{e} + \vect{w}_{2,i}\vect{\gamma} - \nu_i \le 0,  \label{eq:affine:reform:b}
    \end{align}
    \end{subequations}
    is a tractable reformulation of the robust optimization problem \eqref{eq:6} with the affine policy \eqref{eq:affine}.
\end{theorem}
\begin{proof}
    The result follows from a standard reformulation of linear program under conic uncertainty using Cauchy-Schwarz inequality, where we require all inequalities in \eqref{eq:6} to be satisfied constraint-wise under worst-case uncertainty realization.
\end{proof}

\begin{remark}[Comparing two reformulations]
    Theorems \ref{thm:quadratic} and \ref{thm:affine} have shown that, under load uncertainty, imposing affine policy for the second-stage variable $\vect{y}$ in \eqref{eq:6} renders the reformulation convex, whereas only a computationally tractable safe approximation can be derived when quadratic policy is imposed. Because affine policy is nothing but quadratic policy with all $\vect{Q}_i = \mathbbold{0}$, the latter will always outperform the former when the respective problems are solved to optimality; however, because only an approximate formulation can be derived with quadratic policy, the quality of the two are not necessarily comparable. In particular, \eqref{eq:affine:reform} dominates \eqref{eq:quadratic:reform} in the extreme case when $\vect{W}_2 = \mathbbold{0}$. To see this, note that when $\vect{W}_2 = \mathbbold{0}$, by Schur complement lemma, \eqref{eq:quadratic:reform:c}--\eqref{eq:quadratic:reform:d} is equivalent to:
    \begin{subequations} \label{eq:compare}
    \begin{align}
        4\lambda_{0,i} & \ge \frac{\| \vect{w}_{1,i}\tilde{\vect{D}}^{-1}\vect{E} + \vect{w}_{2,i}\vect{K} \|^2}{\lambda_{1,i}} + \sum_{k=1}^T \frac{\| \vect{w}_{2,i}\vect{L}_k - \vect{\theta}_{k,i} \|^2}{\lambda_{k+1,i}}, \label{eq:compare:a} \\
        \lambda_{k,i} & \ge 0, \quad k = 0, \ldots ,T+1. \label{eq:compare:b}
    \end{align}
    \end{subequations}
    Dividing both sides of \eqref{eq:compare:a} by $4$, then adding $\lambda_{k, i}$ for $k \in \{1, \ldots, T+1\}$ to both sides, we get:
    \begin{equation}
    \begin{split}
        \nonumber
        \sum_{k=0}^{T+1} \lambda_{k,i} \geq \frac{\| \vect{w}_{1,i}\tilde{\vect{D}}^{-1}\vect{E} + \vect{w}_{2,i}\vect{K} \|^2 + 4\lambda_{1,i}^2}{4\lambda_{1,i}} \ +\\ \sum_{k=1}^T \frac{\| \vect{w}_{2,i}\vect{L}_k - \vect{\theta}_{k,i} \|^2+ 4\lambda_{k+1,i}^2}{4\lambda_{k+1,i}}
    \end{split}
    \end{equation}
    noting that $(a^2 + 4b^2) / 4b \ge a$ for any $a,b>0$, we deduce:
    \begin{equation} \label{eq:sumlam}
        \sum_{k=0}^{T+1} \lambda_{k,i} \ge
        \| \vect{w}_{1,i}\tilde{\vect{D}}^{-1}\vect{E} + \vect{w}_{2,i}\vect{K} \|
        + \sum_{k=1}^T \| \vect{w}_{2,i}\vect{L}_k - \vect{\theta}_{k,i} \|.
    \end{equation}
    The dominance result follows by plugging \eqref{eq:sumlam} back into \eqref{eq:quadratic:reform:b} and comparing it with \eqref{eq:affine:reform:b}. Of course, the case $W_2 = \mathbbold{0}$ is not very interesting because the problem is no longer an ARO, and we would not need to consider second-stage policy. Nonetheless, we use it to illustrate the nondominance of the reformulation \eqref{eq:quadratic:reform} based on quadratic policy. 
\end{remark}

\begin{remark}[Scenario of deterministic load power]
    It has been shown in \cite{zhen2018computing} that both quadratic and affine policies admit exact convex reformulations when the load power outputs are deterministic. As an alternative proof, notice that when $\vect{\zeta}$ is not present in the maximization problem in \eqref{eq:ref1:b} in Theorem \ref{thm:quadratic}, approximate $S$-lemma ensures the exactness of the convex relaxation in \eqref{eq:quadratic:reform}.
\end{remark}

\section{Simulation Results} \label{sect:sim}

The proposed formulations are tested using data of a real distribution feeder located in the territory of Southern California Edison. The distribution feeder has 126 multiphase nodes with a total of 366 single-phase points of connection. The nominal voltage at the substation is 12 kV, and voltage limits are set to 1.05 p.u. and 0.95 p.u for all nodes. There are 55 uncontrollable loads scattered across the feeder. Dispatchable DERs include 33 PV units, 28 energy storage devices, and 5 HVAC systems. 
Detailed configurations and parameters of the distribution feeder can be found in \cite{bernstein2019real}, and the device model parameters follow \cite{chen2020aggregate}.
MATLAB 2018a is used for all computational experiments. Optimization formulations are modeled using CVX \cite{cvx} and solved with MOSEK \cite{mosek}.

\vspace{-0.1in}

\subsection{Deterministic Model}
In the first experiment, we implement the proposed ellipsoidal inner approximation formulations \eqref{eq:quadratic:reform} and \eqref{eq:affine:reform}, and we compare the results with an existing method in \cite{chen2020leveraging} based on hyperbox inner approximation. We assume that the model is deterministic. Note that in this case the quadratic policy formulation \eqref{eq:quadratic:reform} dominates the one in \eqref{eq:affine:reform} based on affine policy. We set the time horizon to be from 9:00 to 13:00, with 1-hour granularity. The number of time steps is four.

The simulation results on the volumes of the three flexibility region approximations are tabulated in Table \ref{tab:volume:determ}. The proposed ellipsoidal approximations outperform the hyperbox approximation in \cite{chen2020leveraging}. The advantage of the ellipsoidal approach can be attributed to the fact that it is more flexible in capturing the inter-temporal relationship of power outputs, whereas a hyperbox is relatively more rigid.
\begin{table}[]
	\renewcommand{\arraystretch}{1.1}
    \caption[]{Volumes of Flexibility Region Approximation}
    \centering
    \begin{tabular}{cl}
        \toprule
        Method & Volume \\
        \midrule
        Ellipsoid by quadratic policy \eqref{eq:quadratic:reform} & $271.55$ \\
        Ellipsoid by affine policy \eqref{eq:affine:reform} & $217.57$ \\
        Hyperbox approximation \cite{chen2020leveraging} & $96.88$\\
        \bottomrule
    \end{tabular}
    \label{tab:volume:determ}
\end{table}

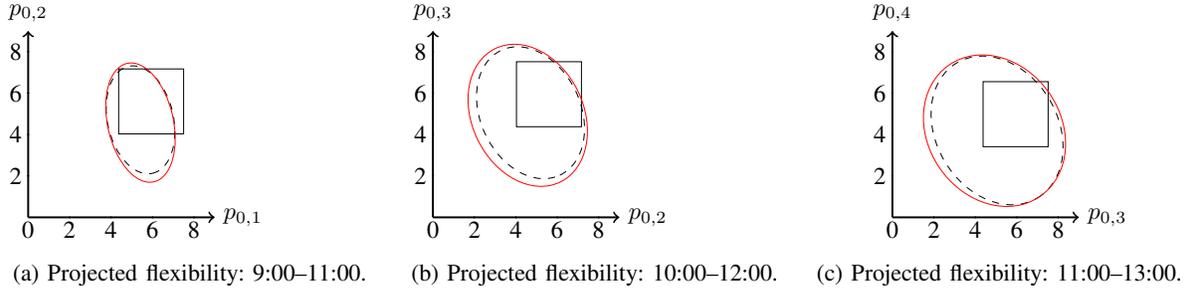
\begin{figure*}
    \centering
    \captionsetup[subfigure]{justification=centering}
     \begin{subfigure}{0.29\textwidth}
        \resizebox{1.5in}{!}{
	    \begin{tikzpicture}[x=0.3cm, y=0.3cm]
	
	    \draw[thick,->] (0,0) -- (9,0) node[right] {$p_{0, 1}$};
	    \draw[thick,->] (0,0) -- (0,9) node[above] {$p_{0, 2}$};
	    
	    \node at (0,-0.6) {0};
	    \node at (2,-0.6) {2};
	    \node at (4,-0.6) {4};
	    \node at (6,-0.6) {6};
	    \node at (8,-0.6) {8};
	    \draw (2,0.05) -- (2,-0.05);
	    \draw (4,0.05) -- (4,-0.05);
	    \draw (6,0.05) -- (6,-0.05);
	    \draw (8,0.05) -- (8,-0.05);
	
	    \node at (-0.6,2) {2};
	    \node at (-0.6,4) {4};
	    \node at (-0.6,6) {6};
	    \node at (-0.6,8) {8};
	    \draw (-0.05,2) -- (0.05,2);
	    \draw (-0.05,4) -- (0.05,4);
	    \draw (-0.05,6) -- (0.05,6);
	    \draw (-0.05,8) -- (0.05,8);
	
	    \draw (4.3662,4.0309) rectangle (7.5009,7.1664);
	    \draw[rotate around={-77.3407:(5.4207, 4.7100)}, dashed] (5.4207, 4.7100) ellipse (2.6459 and 1.5912);
	    \draw[rotate around={12.7438:(5.4182,4.5748)}, color=red] (5.4182,4.5748) ellipse (1.5864 and 2.9316);
	    \end{tikzpicture}
	    }
	    \caption{Projected flexibility: 9:00--11:00.} 
	    \label{fig:proj:1}
     \end{subfigure}
     \begin{subfigure}{0.29\textwidth}
	    \resizebox{1.5in}{!}{
	    \begin{tikzpicture}[x=0.3cm, y=0.3cm]
	
	    \draw[thick,->] (0,0) -- (9,0) node[right] {$p_{0, 2}$};
	    \draw[thick,->] (0,0) -- (0,9) node[above] {$p_{0, 3}$};
	    
	    \node at (0,-0.6) {0};
	    \node at (2,-0.6) {2};
	    \node at (4,-0.6) {4};
	    \node at (6,-0.6) {6};
	    \node at (8,-0.6) {8};
	    \draw (2,0.05) -- (2,-0.05);
	    \draw (4,0.05) -- (4,-0.05);
	    \draw (6,0.05) -- (6,-0.05);
	    \draw (8,0.05) -- (8,-0.05);
	
	    \node at (-0.6,2) {2};
	    \node at (-0.6,4) {4};
	    \node at (-0.6,6) {6};
	    \node at (-0.6,8) {8};
	    \draw (-0.05,2) -- (0.05,2);
	    \draw (-0.05,4) -- (0.05,4);
	    \draw (-0.05,6) -- (0.05,6);
	    \draw (-0.05,8) -- (0.05,8);
	
	    \draw (4.0309, 4.3763) rectangle (7.1664, 7.5147);
	    \draw[rotate around={-64.8775:(4.7100, 5.0507)}, dashed] (4.7100, 5.0507) ellipse (3.3293 and 2.4168);
	    \draw[rotate around={-64.0639:(4.5748, 4.9272)}, color=red] (4.5748, 4.9272) ellipse (3.5833 and 2.6877);
		
	    \end{tikzpicture}
	    }	
	    \caption{Projected flexibility: 10:00--12:00.} 
	    \label{fig:proj:2}
     \end{subfigure}
     \begin{subfigure}{0.29\textwidth}
        \centering
	    \resizebox{1.5in}{!}{
	    \begin{tikzpicture}[x=0.3cm, y=0.3cm]
	
	    \draw[thick,->] (0,0) -- (9,0) node[right] {$p_{0, 3}$};
	    \draw[thick,->] (0,0) -- (0,9) node[above] {$p_{0, 4}$};
	    
	    \node at (0,-0.6) {0};
	    \node at (2,-0.6) {2};
	    \node at (4,-0.6) {4};
	    \node at (6,-0.6) {6};
	    \node at (8,-0.6) {8};
	    \draw (2,0.05) -- (2,-0.05);
	    \draw (4,0.05) -- (4,-0.05);
	    \draw (6,0.05) -- (6,-0.05);
	    \draw (8,0.05) -- (8,-0.05);
	
	    \node at (-0.6,2) {2};
	    \node at (-0.6,4) {4};
	    \node at (-0.6,6) {6};
	    \node at (-0.6,8) {8};
	    \draw (-0.05,2) -- (0.05,2);
	    \draw (-0.05,4) -- (0.05,4);
	    \draw (-0.05,6) -- (0.05,6);
	    \draw (-0.05,8) -- (0.05,8);
	
	    \draw (4.3763, 3.4115) rectangle (7.5147, 6.5521);
	    \draw[rotate around={-59.2044:(5.0507, 4.1966)}, dashed] (5.0507, 4.1966) ellipse (3.7974 and 2.9357);
	    \draw[rotate around={-55.2511:(4.9272, 4.1843)}, color=red] (4.9272, 4.1843) ellipse (3.8563 and 3.2049);
		
	    \end{tikzpicture}
	    }	
	    \caption{Projected flexibility: 11:00--13:00.} 
	    \label{fig:proj:3}
     \end{subfigure}
     \caption{Projection of flexibility region approximations. (hyperbox: black solid line; affine policy: black dashed line; quadratic policy: red solid line.)}\vspace{-12pt}
     \label{fig:projection}
\end{figure*}

To visualize the approximated flexibility region, we project the resulting approximations on the planes. Let $p_{0,t}$ denote the substation power at time interval ($t$+8):00--($t$+9):00. Fig. \ref{fig:projection} shows the flexibility region approximations projected on three planes representing $p_{0,1}$--$p_{0,2}$, $p_{0,2}$--$p_{0,3}$, and $p_{0,3}$--$p_{0,4}$. It is quite clear from the figures that the approximations given by the proposed ellipsoidal characterizations are less conservative. Both ellipsoidal approximations almost dominate the hyperbox one in Fig. \ref{fig:proj:3}; however, note that a point is feasible for the hyperbox approximation if and only if it is feasible for each projected dimension, whereas this is not the case for the proposed ellipsoidal approximations. Also note that even though quadratic policy leads to an approximated region with no less volume than the affine policy, the flexibility region given by the former is not necessarily a super set of the latter, as shown in Fig. \ref{fig:proj:1}, where the dashed ellipses stretch slightly outside of the red ones on the right side of the two ellipses. 

\subsection{Model with Uncertainty}

In this section, we compute the flexibility region of the model studied in the last section under different levels of uncertainty. Although \eqref{eq:quadratic:reform} is computationally intensive, formulation \eqref{eq:affine:reform} based on affine policy yields satisfactory results within a reasonable time. Table \ref{tab:uncertain} shows the results of the volume of flexibility region characterized by \eqref{eq:affine:reform} under increasing levels of uncertainty. It shows that the computation time slightly increases with level of uncertainty but remains within 1 minute even for the most challenging case, where the aggregate flexibility region almost vanishes.
\begin{table}[]
	\renewcommand{\arraystretch}{1.1}
    \caption[]{Flexibility by Affine Policy Under Uncertainty}
    \centering
    \begin{tabular}{cll}
        \toprule
        Uncertainty level ($\delta$) & Volume & Time (sec.)  \\
        \midrule
        $5\%$  & $148.88$ & $45.64$ \\
        $10\%$ & $97.25$  & $41.52$ \\
        $15\%$ & $57.20$  & $48.09$ \\
        $20\%$ & $30.20$  & $48.73$ \\
        $25\%$ & $12.67$  & $49.59$ \\
        $30\%$ & $2.78$   & $52.91$ \\
        \bottomrule
    \end{tabular}
    \vspace{-10pt}
    \label{tab:uncertain}
\end{table}

\section{Conclusions} \label{sect:conclusions}

This paper proposed novel formulations for the inner approximation of the distribution system flexibility region. The approaches use efficient policy-based approximations of the ARO formulation of computing inscribed MVE of a polytopic projection. The proposed formulations are the first to comprehensively address the issues of network awareness, load uncertainty, and time coupling. Simulation results validated that the deterministic version of the proposed results outperform the state of the art in terms of approximation quality. Some future research directions include considering the coupled real-reactive power flexibility region as well as alternative uncertainty modeling.
\balance

\bibliographystyle{IEEEtran}
\bibliography{bibtex/IEEEabrv,bibtex/flex}

\end{document}